\documentclass[a4paper,reqno,index]{amsart}
\reversemarginpar
\usepackage{amsmath, amsfonts, amsthm, amssymb}
\usepackage[maxbibnames=9]{biblatex}
\usepackage{mathrsfs}
\usepackage{fancyhdr}
 
\usepackage{enumerate}   
\usepackage{blkarray}
\usepackage{mathrsfs}
\usepackage[title]{appendix}
\usepackage{amsthm}
\usepackage{xcolor}
\usepackage{graphicx}
\usepackage[normalem]{ulem}
\usepackage{morefloats}
\usepackage{float}

\usepackage{hyperref}
\usepackage{mathtools}

\usepackage[normalem]{ulem} 

 \numberwithin{equation}{section}

\allowdisplaybreaks

\makeindex

\hypersetup{
    bookmarks=true,         
    unicode=false,          
    pdftoolbar=true,        
    pdfmenubar=true,        
    pdffitwindow=false,     
    pdfstartview={FitH},    
    pdftitle={My title},    
    pdfauthor={Author},     
    pdfsubject={Subject},   
    pdfcreator={Creator},   
    pdfproducer={Producer}, 
    pdfkeywords={keyword1, key2, key3}, 
    pdfnewwindow=true,      
    colorlinks=false,       
    linkcolor=green,          
    citecolor=green,        
    filecolor=green,      
    urlcolor=green,           
    urlbordercolor={1 1 1}  
}

\newtheorem{theorem}{Theorem}[section]
\newtheorem*{theorem*}{Theorem}
\newtheorem*{conjecture*}{Conjecture}
\newtheorem{lemma}[theorem]{Lemma}

\newcommand{\is}{\hspace{2pt}}

\newcommand{\DD}{\mathbb{D}}

\newcommand{\eps}{\varepsilon}
\renewcommand{\epsilon}{\varepsilon}
\renewcommand{\phi}{\varphi}
\renewcommand{\kappa}{\varkappa}
\newcommand{\cC}{\mathcal C}

\newtheorem{corollary}[theorem]{Corollary}
\theoremstyle{definition}
\newtheorem{definition}[theorem]{Definition}
\newtheorem{question}[theorem]{Question}

\theoremstyle{remark}
\newtheorem{remark}[theorem]{Remark}

 \usepackage{etoolbox}

\begin{document}

\newcommand{\C}{\mathbb{C}}

\newcommand{\N}{\mathbb{N}}
\newcommand{\T}{\mathbb{T}}

\newcommand{\al}{\alpha}
\newcommand{\be}{\beta}
\newcommand{\wt}{\widetilde}
\newcommand{\cK}{\mathcal{K}}

\newcommand{\de}{\delta}
\newcommand{\ga}{\gamma}
\newcommand{\la}{\lambda}
\newcommand{\tht}{\theta}
\newcommand{\si}{\sigma}

\newcommand{\fD}{\mathfrak{D}}
\newcommand{\DY}{\color{magenta}}

\theoremstyle{definition}

\theoremstyle{remark}

\title[\footnotesize Operators with small Kreiss constants]{\large Operators with small Kreiss constants}

\author[\footnotesize CHALMOUKIS]{NIKOLAOS CHALMOUKIS}
\address{Department of Mathematics, University of Milano - Bicocca, Milan, Italy}
\email{nikolaos.chalmoukis@unimib.it} 
\author[\footnotesize TSIKALAS]{GEORGIOS TSIKALAS}
\address{Department of Mathematics, Vanderbilt University, Nashville, TN}
\email{georgios.tsikalas@vanderbilt.edu} 

\author[\footnotesize YAKUBOVICH]{DMITRY YAKUBOVICH}
\address{Departamento de Matem\'aticas, Universidad Aut\'onoma de Madrid, Canto-blanco, 28049 Madrid, Spain}
\email{dmitry.yakubovich@uam.es}

\thanks{
The first author is a member of the INdAM group GNAMPA. The second and third authors  were partially supported by Plan Nacional  I+D grant no. PID2022-137294NB-I00, Spain.
This work has also been supported by the Madrid Government (Comunidad deMadrid – Spain) under the multiannual Agreement with UAM in the line for the Excellence of the University Research Staff in the context of the V PRICIT (Regional Programme of Research and Technological Innovation).
}

\subjclass[2020]{47A20, 47A65} 
\keywords{Power-bounded operators, resolvent, Kreiss Matrix Theorem, similarity, double-layer potential}
\small
\begin{abstract}
    \small
       We investigate matrices satisfying the Kreiss condition
       \[\|(zI-T)^{-1}\|\le\cfrac{K}{|z|-1}, \hspace*{0.3 cm} |z|>1, \]
       with $K$ lying arbitrarily close to $1.$ We provide lower bounds for the power growth of such matrices, which complement and refine related estimates due to Nikolski and Spijker-Tracogna-Welfert.

       We also study operators that satisfy a variant of the above Kreiss condition where $K$ is replaced by 
     $1+\eps(|z|)$, where the positive continuous function  $\eps(|z|)$ tends to $0$ as $|z|\to 1^+.$
       We show that, 
       if the spectrum of $T$ touches the unit circle only at a single point and the   
       resolvent of $T$ satisfies a growth restriction along the unit circle, it is possible to choose $\eps$ so that this Kreiss-type condition guarantees similarity to a contraction. At the core of our proof lies a positivity argument involving the double-layer potential operator. Counterexamples related to less restrictive choices of $\eps$ are also provided.

\end{abstract}
\maketitle

\section{Introduction}
\subsection{The Kreiss condition} The Kreiss Matrix Theorem,
a well-known and important fact in applied matrix analysis, concerns matrices satisfying the following resolvent condition.
\begin{definition}
An $N\times N$ matrix $A$ is said to satisfy the \textit{Kreiss resolvent condition} if the spectrum of $A$ is contained in $\{|z|\le 1\}$ and there exists a constant $K\ge 1$ such that
\begin{equation}\label{def}
    \|(zI-A)^{-1}\|\le \cfrac{K}{|z|-1}, \hspace{0.5 cm} \forall |z|>1. 
\end{equation}
The least constant $K$ such that \eqref{def} holds is termed the \textit{Kreiss constant} of $A.$
\end{definition}
Kreiss \cite{Kreiss62} proved that every matrix that satisfies the above  condition is power-bounded. More specifically, for every $N\times N$ matrix $A$ with Kreiss constant $K$, there exists $M(N, K)>0$ with 
\[\|A^n\|\le M(N, K), \hspace{0.5 cm} \forall n\ge 1.\]
 In applications, however, the dimension $N$ may be very large (see e.g. \cite{Dorsselaer93}), so it is of interest to look for sharp constants $M(N, K)$.

 Indeed, in the years after Kreiss published his result, various authors made 
 successive improvements
 that reduced the value of $M(N, K)$ (see \cite[Section 18]{pseudospectrabook} for the full history and references). Tadmor \cite{Tadmor81} was the first to show that  $M(N, K)\le CKN$  for some numerical constant $C>0$. Finally, Spijker \cite{Spijkersharp}, making crucial use of a Bernstein-type inequality for rational functions, proved that $M(N, K)\le eKN$ and so, if $A$ is any $N\times N$ matrix with Kreiss constant $K,$ then 
\begin{equation} \label{spijk}
\|A^n\|\le e K N, \hspace{0.5 cm} \forall n\ge 1.     
\end{equation}
LeVeque and Trefethen \cite{LeVequeTref}  showed that \eqref{spijk} is sharp in the sense that, for every $N\ge 1,$ there exists an $N\times N$ matrix $A_N$ with Kreiss constant  $K_N$ such that
\[\frac{\sup_{n\ge 1}\|A^n_N\|}{K_N N}\to e \hspace{0.5 cm} \text{ as } N\to\infty. \]
\par This sequence $K_N$ is, however, unbounded, so it is natural to ask 
to what extent the Kreiss-Spijker estimate is sharp for an arbitrary \textit{fixed} Kreiss constant. To this end, define $P(N, K)$ to be the smallest number such that 
every $N\times N$ matrix $A$ with Kreiss constant 	less than or equal to
$K$ satisfies
\begin{equation} \label{Pdef}
\|A^n\|\le P(N, K), \hspace{0.5 cm} \forall n\ge 1.     
\end{equation}
By \eqref{spijk}, we have $P(N, K)\le eKN.$ What about lower bounds? The first results of this type are due to McCarthy-Schwartz \cite{McCarthySchwartz} and  van Dorsselaer-Kraaijevanger-Spijker \cite{Dorsselaer93}, who established lower bounds for $P(N, K)$ that were logarithmic in $N$.  Later, Spijker, Tracogna and Welfert \cite{STW}  almost bridged the gap with \eqref{spijk} by showing that, given $\epsilon>0,$ there exists $K=K(\epsilon)>\pi+1$ and $C=C(\epsilon)>0$  such that
\begin{equation} \label{spicktracwel}
   P(N, K) \ge C N^{1-\epsilon}, \hspace{0.5 cm} \forall N\ge 1.  
\end{equation}
Their construction was subsequently refined by Nikolski \cite{Nikolski}, who showed that, for any $K>\pi +1$ and $N\ge 1,$
\begin{equation}\label{nik}
   P(N, K)\ge \frac{N^{1-\frac{2}{K}}}{5}.  
\end{equation}
 More recently, the search for non-power-bounded operators that satisfy various Kreiss- and Ces\`aro-type conditions has focused on weighted shifts; see e.g.  \cite{AbsolCesaromuller, KreissanduniKreiss, CohenCunyEisnerLin}.     
    
\subsection{Small Kreiss constants} Unfortunately, all of the above-mentioned papers fall short of providing lower bounds for $P(N, K)$ for $K$ close to $1.$ In fact, it appears that the only time non-power-bounded operators with Kreiss-type constants close to $1$ were considered in the literature was in a paper of McCarthy \cite{McCarthy71lost} (see also the remark in \cite[p. 374]{Nikolski}), who proved that, for any $\epsilon>0,$ there exists $C=C(\epsilon)$ with
\[
P(N, 1+\epsilon)\ge C \sqrt{\log \log (N+1)}, \hspace{0.4 cm} N\ge 1. 
\]

Our first main result  is a substantial improvement of McCarthy's estimate.  Our proof is based on the construction of a certain family of  weighted shifts;  the novel ingredient in
our approach is the consideration of \textit{matrix weights}. More precisely, given $\epsilon>0$ and $m\in\mathbb{N}$, we show that there exists a constant $C=C(\epsilon, m)$ such that for all $N\ge 1,$
\begin{equation} \label{main}
    P(N, 1+\epsilon)\ge C \log^m (N).
\end{equation}
 
\noindent 
The exact same lower bound holds for the (more restrictive) \textit{uniform Kreiss condition} (see Section \ref{pre} for the definition); we prove this as Theorem \ref{uniformthm}.  
 In fact, our construction yields  lower bounds for even more general Kreiss- and Ces\'aro-type conditions. To be more precise, define $P_{AC}(N, K)$ as the supremum of the power bounds of all $N\times N$ matrices that are \textit{absolutely Ces\`aro bounded} (see Section \ref{pre}) with constant at most $K.$

%
%

\begin{theorem}\label{cesaromain}
	\begin{itemize}
	\item[(i)] Let $\epsilon>0$ and $m\in\mathbb{N}$. There exists a constant $C=C(\epsilon, m)$ such that for all $N\ge 1,$
	\begin{equation}\label{estim-i}
	P_{AC}(N, 1+\epsilon)\ge C \log^m (N).
	\end{equation}

     \item[(ii)] Let $\alpha\in (0,\frac 14)$.
     There exist $M=M(\alpha)>0$ and $\eps_0\in (0,1]$ such that, whenever $0<\eps<\eps_0$ and $\epsilon \log N\ge M$, one has 
	\begin{equation}\label{estim-ii}     
		P_{AC}(N, 1+\epsilon)\ge \exp\bigg(\,\frac{\alpha\log^2 \log N^{\eps}}{\log\log\log N^{\eps}}\bigg). 
    \end{equation}
\end{itemize}
\end{theorem}
%
%
%
%
%
%
\noindent Since~\eqref{estim-i} says that 
\[
P_{AC}(N, 1+\epsilon)\ge \exp (m\log\log N+\log C), 
\] 
we can assert that~\eqref{estim-ii} is obtained from 
~\eqref{estim-i} by replacing the function 
$\exp(m\log \log N)$ with 
a function that grows slightly slower than $\exp(\alpha\log^2\log N)$. So, the estimate~\eqref{estim-ii} is considerably sharper than~\eqref{estim-i}. 

\par 

We also observe that  Theorem \ref{cesaromain} does not immediately follow from \eqref{main}; indeed, while it is well-known that absolute Ces\'aro boundedness implies the Kreiss condition, it could happen that the Kreiss constant of a matrix exceeds its absolute Ces\`aro constant.\par  A  lower bound corresponding to the \textit{strong Kreiss condition} can be found in \eqref{strongkreissineq}; see Remark \ref{further2}.
 
\subsection{The similarity problem}
We now move from matrices to infinite-dimensional linear operators.
Let $H$ be a complex Hilbert space and let $T\in\mathcal{B}(H)$ with $\sigma(T)\subset \overline \DD$. The Kreiss condition 
\[\label{Kreiss}
\|(\la-T)^{-1}\|\le \frac C{|\la|-1}, \quad \quad 1<|\la|\le 1+\nu
\]
(where $\nu>0$) no longer implies power-boundedness in this setting; classical counterexamples are contained in \cite{Foguelcounterex, HalmosFoguel}. On the other hand, if $C=1$, then $T$ is a $\rho$-contraction for some $\rho>1$, and so it is similar to a contraction (and thus power-bounded), no matter the value of $\nu$ (see e.g. \cite{Foiasbook}).  
\par 
In the second half of our paper, we examine several   conditions on $T$ that lie in-between the above $C=1$ and $C>1$ regimes, and determine whether they imply  similarity to a contraction. We point out that similarity to a contraction 
(and, more generally, $K$-spectral estimates)
has
already been linked to resolvent conditions several times in the literature, see e.g. \cite{BouabdLeMerdy, CassiergeneralizedToepl,deLauben,TestsforKsp, Gamal}. For more references and recent developments concerning the similarity problem, especially in relation to operator semigroups, see \cite{mazatomilov}. One of the conditions we are interested in, naturally motivated by the consideration of ``small" Kreiss constants, is 
\begin{equation}\label{eps-resolv-estim}
\|(\la-T)^{-1}\|\le \frac {1+\eps(|\la|-1)}{|\la|-1} \quad \quad 1<|\la|\le 1+\nu,
\end{equation}
where $\eps:(0,\nu)\to(0,+\infty)$ is a continuous function such that $\eps(0^+)=0$. We use our weighted shifts from the proof of Theorem \ref{cesaromain} to give an explicit choice for $\eps$ such that ~\eqref{eps-resolv-estim} is satisfied but $T$ is not similar to a contraction (in fact, not even power-bounded); this is Theorem \ref{resolvcounter}. \par 
We also consider operators $T$ such that $\lim_{n\to\infty}\|T^n\|=1$. Notice that if $\|T^m\|\le 1$ 
for some $m\in\N$, then $T$ is similar to a contraction, see ~[\cite{Halmos-ten}, discussion of Problem~6].  
It is easy to see that the condition $\lim_{n\to\infty}\|T^n\|=1$ implies~\eqref{eps-resolv-estim} for a certain function 
$\eps(x)$ vanishing at $0^+$. By adapting Pisier's example of an operator that is polynomially bounded but not similar to a contraction \cite{Pisier-result}, we
 show
\begin{theorem}\label{thm-Pisier-counterex}
	Let $\{\be_n:n\ge 1\}$ be a positive sequence such that 
	$\lim_{n\to\infty}\be_n=0$, 
	$\sum_n \be_n/n=\infty$ and the sequence $\{n^{-2}\be_n\}$ is decreasing. Then there 
	exists a Hilbert space operator $F$ such that 
	\[
	\|F^n\|\le 1+\sqrt{\be_n}, 
	\]
	but $F$ is not similar to a contraction. 	 
\end{theorem}
\noindent In particular, one can take $\be_n=1/\log(n+1)$. In this case, the operator $F$ will satisfy \eqref{eps-resolv-estim} with $\eps(x)\simeq [\log(1/x)]^{-1/2}$; see Remark \ref{epsconversion}.

\par 

It seems to be unknown whether there exists a function $\eps(x)$, which goes (rapidly) 
to $0$ as $x\to 0^+$, such that~\eqref{eps-resolv-estim}
guarantees the similarity of $T$ to a contraction. 
However, we are able to give positive results of this type under some stronger assumptions on~$T$. For convenience, we introduce the following definition. 

\begin{definition}\label{v-type-def}
Let $r:[-\de_-,\de_+]\to [1,+\infty)$ be a continuous function, where 
$\de_-,\de_+\in(0,\pi)$. Consider the corresponding curve 
 $\ga(\theta)=r(\theta)e^{-i\theta}$,
$\theta\in  [-\de_-,\de_+],$ in the complex plane. We will say that 
$\ga$ {\it is a v-type curve} if the following 
conditions hold. 
\begin{itemize}
	\item[(a)]$r(\cdot)$ strictly increases on $[0, \de_+]$ and strictly decreases on $[-\de_-, 0]$; 	
	
	\item[(b)] $r(0)=1$ and $r(-\de_-)=r(\de_+)>1$;
	
	\item[(c)] ${\displaystyle \int^{\de_+}_{-\de_{-}}\cfrac{1}{r(\theta)-1}\ d\theta<\infty}$;
%
\end{itemize}
\end{definition}


%

\noindent A simple example of a function $r(\cdot)$ which gives rise to a 
v-type curve is $r(\theta)=1+|\theta|^p,$ where $0<p<1$. 
If $\ga$ is a v-type curve, then 
its union with the arc 
$\{  r(\de_+)e^{-i\tht}; \; -\de_-\le \tht\le \de_+\}$ 
is a Jordan curve. We will say that the closed domain, bordered by this curve, is 
{\it the set determined by the v-type curve $\ga$}. Here is our second main result.
 
%

%
%

\begin{theorem}\label{genKreiss2}
	Let $T\in\mathcal{B}(H)$ with $\sigma(T)\subset \DD\cup\{1\}$, with $\ga:[-\de_-,\de_+]\to \C$ a v-type curve. 
%
%
%
Assume that $T$ satisfies the Kreiss condition and, moreover, $\varepsilon: (0, \nu]\to (0, \infty)$ is a continuous increasing function 
	such that $\eps(0^+)=0$ and  the resolvent estimate~\eqref{eps-resolv-estim} holds for all $\la$ 
    in the set determined by $\ga$ (excluding $\la=1$). 
	Set $r(\tht)=|\ga(\tht)|$. If 
\[
\displaystyle \int_{-\de_-}^{\de_+}\cfrac{\varepsilon(r(\theta)-1)}{r(\theta)-1}\,\|(e^{i\theta}-T)^{-1}\|^2\, d\theta <\infty, 
\]
then $T$ is similar to a contraction. 

\end{theorem}

%
%

\noindent In other words, suppose that $\si(T)$ is contained in $\DD\cup\{1\}$ 
and one knows an estimate of $\|(\la-T)^{-1}\|$ for $\la$ 
on the unit circle $\mathbb T$. 
Then, given a v-type curve $\ga$, one can exhibit a positive function $\eps$ such that whenever the estimate~\eqref{eps-resolv-estim} holds on the set determined by this curve,  $T$ is similar to a contraction.

\par

Let us state a particular case of this theorem for 
Ritt 
operators. 
Recall that $T$ is called a Ritt operator if 
there exists $C>0$ such that 
\[ \|(\lambda-T)^{-1}\|\le \frac{C}{|\la-1|}\]
for all $|\la|>1$. 
Any Ritt operator satisfies $\si(T)\subset\DD\cup\{1\}$. 
The point is that, since we can estimate the resolvent of $T$ along $\mathbb{T}$ if $T$ is Ritt, it is possible to make a uniform choice of $\varepsilon$ for such operators. 
Indeed, there exists $M>0$ such that 
\begin{equation} \label{reest}
\|(e^{i\theta}-T)^{-1}\|\le \frac{M}{|\theta|}    
\end{equation}
 for all $\theta$ (see e.g. \cite[Section 1.2]{schwenntadritt}). 
Thus, in view of Theorem \ref{genKreiss2}, we immediately obtain
\begin{corollary}
Let $T\in\mathcal{B}(H)$ be a Ritt operator. Assume there exists $\varepsilon: (0, \nu]\to(0, \infty)$ continuous and increasing such that $\epsilon(0^+)=0$ and~\eqref{eps-resolv-estim} holds in the set determined by some v-type curve $\ga$. 
%
%
If $\varepsilon$ satisfies 
 \[
 \int_{-\de_-}^{\de_+}\frac{\varepsilon(r(\theta)-1)}{|\theta|^2(r(\theta)-1)}\ d\theta<\infty,
 \]
 where $r(\tht)=|\ga(\tht)|$, then $T$ is similar to a contraction. 
 \end{corollary}
 \noindent We note that Ritt operators, while power-bounded, are not, in general, similar to contractions \cite[Proposition 5.2]{LMsimprob}.

\subsection{Outline}  Our paper is organized as follows: in Section \ref{pre}, we briefly go over the definitions and the relations between those Kreiss- and Ces\'aro-type conditions that are relevant to this work. We also record an inequality from \cite{AbsolCesaromuller} that will 
be very useful in the proof of Theorem \ref{cesaromain}.
 Section \ref{mainproof} contains the proofs of Theorem \ref{cesaromain} and Theorem \ref{uniformthm}, and, in addition, a lower bound corresponding to the strong Kreiss condition (Remark \ref{further2}). Theorem \ref{genKreiss2} is proved in Section \ref{genKreiss2sec}. Section \ref{counterexsection} contains the proofs of Theorems \ref{resolvcounter} and \ref{thm-Pisier-counterex}. Section \ref{questions} concludes with some open questions related to our investigations.

\section{Kreiss- and Ces\`aro-type conditions}\label{pre}

  The following conditions often appear in the literature in conjuction with the Kreiss Matris Theorem.
\begin{definition}
 Let $T\in\mathcal{B}(H).$ We say that $T$  \begin{itemize}
     \item[(i)] satisfies the \textit{uniform Kreiss condition} if there exists $K_U>0$ such that 
\[\Big\|\sum_{k=0}^nz^{-k-1}T^k\Big\|\le \frac{K_U}{|z|-1},\hspace{0.2 cm} \text{ for all } |z|>1 \text{ and } n=0, 1, 2,\dots;\]

     \item[(ii)] satisfies the \textit{strong Kreiss condition} if there exists $K_{S}>0$ such that 
\[
\big\|(zI-T)^{-k}\big\|\le \frac{K_S}{(|z|-1)^k},\hspace{0.2 cm} \text{ for all } |z|>1 \text{ and } k=1, 2,\dots; 
\]  
     \item[(iii)] is \textit{Ces\`aro bounded} if there exists $K_C>0$ such that 
     \[\|M_n(T)\|:=\bigg\|\frac{1}{n+1}\sum_{k=0}^nT^k\bigg\|\le K_C,\hspace{0.4 cm} n= 1, 2, 3, \dots;\]
     \item[(iv)] is \textit{absolutely Ces\`aro bounded} if there exists $K_{AC}>0$ such that 
\[\sup_{n\in\mathbb{N}}\frac{1}{n}\sum_{j=1}^n\|T^jx\|\le K_{AC}\|x\|, \hspace{0.4 cm} \forall x\in H. \]
 \end{itemize}
 \end{definition}
Clearly, power-boundedness implies all of the above conditions. Absolute Ces\`aro boundedness and the strong Kreiss condition are both strictly stronger than the uniform Kreiss condition, see \cite{ AbsolCesaromuller, GomilkoZemanek, Mont}, while the two are independent of one another \cite{CohenCunyEisnerLin}. Further, the uniform Kreiss condition is equivalent \cite{Mont} to the existence of a constant $C>0$ such that, for all $n\ge 1$ and $|\lambda|=1,$
\begin{equation}\label{unkreisscharact}
 \|M_n(\lambda T)\|\le C,   
\end{equation}
and is thus stronger than Ces\'aro boundedness (and also, trivially, stronger than the Kreiss condition). The Kreiss condition and Ces\`aro boundedness are independent \cite{SuciuZem}. Finally, while the strong Kreiss condition does not imply power-boundedness \cite{CohenCunyEisnerLin}, we observe that the similar-looking condition 
\[
\big\|T^k(zI-T)^{-k}\big\|\le \frac{K'}{(|z|-1)^k},\hspace{0.2 cm} \text{ for all } |z|>1 \text{ and } k=1, 2,\dots,  
\]  
is, in fact, equivalent to power-boundedness 
\cite{ Gibson, Packel}.  \par 
The definition of $P(N, K)$ can be adapted to all of the above conditions. For instance, we may define $P_U(N, K)$ to be the supremum of the power bounds of all $N\times N$ matrices that satisfy the uniform Kreiss condition  with constant $K,$ with analogous definitions for $P_{S},$ $P_{AC}$ and $P_C.$ We observe that, trivially, 
\[P(N, K)\ge P_{U}(N, K)\ge P_{S}(N, K) \hspace*{0.3 cm}  \text{ and }  \hspace*{0.3 cm}  P_{C}(N, K)\ge P_{AC}(N, K), \]
for all $N\ge 1$ and $K\ge 1.$

An interesting example of an operator that is absolutely Ces\'aro but not power bounded is given in [\cite{AbsolCesaromuller}, Theorem 2.1] (we only record the $\ell^2$-version here):
\begin{theorem}\label{mullerthm}
    Let $T$ be the weighted backward shift on $\ell^2(\mathbb{N})$ defined by $Te_1=0$ and $Te_k=w_ke_{k-1}$ for $k>1.$ If $w_k=\big(\frac{k}{k-1}\big)^a$, with $0<a<1/2,$ then $T$ is not power-bounded but is absolutely Ces\`aro bounded with constant 
    \[K_{AC}\le  \sqrt{2(1-2a)^{-1}+2}.\]
\end{theorem}\noindent
Thus, for any $\ell^2$ sequence of positive numbers $\{b_k\}$ and $0<a<1/2,$ we have 
 \begin{equation}\label{mullereq}
\sup_{N\in\mathbb{N}}\frac{1}{N}\sum_{k=1}^{N}\sqrt{\sum_{j=1}^{\infty}\bigg(\frac{k+j}{j}\bigg)^{2a}b^2_{k+j}} \le C(a) \sqrt{\sum_{k=1}^{\infty}b^2_k},
 \end{equation}
 with $C(a)=\sqrt{2(1-2a)^{-1}+2}.$ 
This inequality will play an important role in the proof of Theorem \ref{cesaromain}.

\section{Lower bounds for small Kreiss constants}\label{mainproof}
\subsection{Lower bounds for the absolute Ces\`aro condition} \label{cesarosubsection}

     Fix $m\ge 1.$ For $k\ge 2$ and $c>0,$ define the upper triangular
     $2^m\times 2^m$ matrix $A_{m, k}(c)$ inductively as follows: set $A_{0, k}(c)=1$, and, assuming the $2^p\times 2^p$ matrix $A_{p, k}(c)$ has been defined for $0\le p<m,$ set 
     \begin{equation} \label{adef}
         A_{p+1, k}(c)=\begin{bmatrix}
A_{p, k}(c) & c\log \big(\frac{k}{k-1}\big)I_{2^p} \\
0 & A_{p, k}(c)
\end{bmatrix},\
     \end{equation}
where $I_{2^p}$ is the $2^p\times 2^p$ identity matrix. 
Further, define the weighted backwards shifts $T_{p, c}: \ell^2\big(\mathbb{C}^{2^p}\big)\to \ell^2\big(\mathbb{C}^{2^p})$, with $0\le p\le m$,
by 
\[
T_{p,c}(h_1, h_2, h_3, \dots) = \big(A_{p, 2}(c)h_2, A_{p, 3}(c)h_3, A_{p, 4}(c)h_4, \dots\big).\]
Notice that $T_{0, c}$ coincides with the usual (unweighted) backwards shift on $\ell^2.$

In what follows, we will deal with the powers of $T_{p+1, c}$. Notice that 
\begin{align*}
	& T_{p+1, c}^nh \\
	&=\Bigg(\bigg[\prod_{k=2}^{n+1}A_{p+1, k}(c)\bigg]h_{n+1}, \bigg[\prod_{k=2}^{n+1}A_{p+1, k+1}(c)\bigg]h_{n+2}, \bigg[\prod_{k=2}^{n+1}A_{p+1, k+2}(c)\bigg]h_{n+3},  \dots\Bigg) 
\end{align*} 
for $j\ge 0$ and 
\begin{equation} \label{variemaiii}
	\prod\limits_{k=2}^{n+1}A_{p+1, k+j}(c)=
	  \begin{bmatrix}
		\prod\limits_{k=2}^{n+1}A_{p, k+j}(c) & c B_{p, n, j+1}(c)\\
		0  & \prod\limits_{k=2}^{n+1}A_{p, k+j}(c)
	  \end{bmatrix}, 
 \end{equation}
where 

\[
B_{p, n,j+1}(c)
:=\sum\limits_{k=2+j}^{n+j+1}\log\Big(\frac{k}{k-1}\Big) 
\prod_{	\substack{j+2\le r\le n+j+1\\ r\neq k}}
 A_{p, r}(c). 
\]

\begin{lemma}\label{AA}
	For any $1\le p\le m$, integers $s\ge 1$ and $2\le k_1<k_2<\dots <k_s$ and constant $0\le c\le 1$, 
	\[
	\Big\|\prod_{i=1}^sA_{p, k_i}(c)\Big\|\le 2^{p-1}\Big[1+\log^p\, \frac{k_s}{k_1-1} \Big].
	\]
	%
	%
\end{lemma}
\begin{proof}
	Fix integers $2\le k_1<k_2<\dots <k_s$ and constant $0\le c\le 1$. We proceed by induction. For $p=1$, 
\[
	\Big\|\prod_{i=1}^sA_{1, k_i}(c)\Big\|= \Bigg\| \begin{bmatrix}
		1 & c\,\log \,\prod\limits_{i=1}^s\frac{k_i}{k_i-1}   \\
		0   &   1
	\end{bmatrix} \Bigg\| 
	\le  1+ c\log \,\prod\limits_{i=1}^s\frac{k_i}{k_i-1} 
	\le  1+ \log \,\frac{k_s}{k_1-1}\, ,
\]
%
	as desired. For the inductive step, assume the conclusion of the lemma holds for 
	a certain value of $p$, $1\le p<m$. Then,
	%
	%
	%
	%
	\begin{align*}
		\Big\|\prod_{i=1}^s  A_{p+1, k_i}(c)& \Big\| 
		= \ 
		\Bigg\|\begin{bmatrix}
			\prod\limits_{i=1}^sA_{p, k_i}(c) & c\sum\limits_{i=1}^s \Big(\log\,\cfrac{k_i}{k_i-1}\Big) \prod\limits_{1\le j\neq i}^s  A_{p, k_j}(c) \\
			0   &   \prod\limits_{i=1}^sA_{p, k_i}(c)
		\end{bmatrix} \Bigg\| \\ 
		&\le \  
		\Big\|\prod_{i=1}^sA_{p, k_i}(c)\Big\|
		+c\, \Big\|\sum\limits_{i=1}^s\Big(\log\, \frac{k_i}{k_i-1}\Big) \prod\limits_{1\le j\neq i}^s  A_{p, k_j}(c)\Big\| \\
		& \le  \
		2^{p-1}\Big[1+\log^p\, \frac{k_s}{k_1-1} \Big]+2^{p-1}\sum_{i=1}^s\Big(\log\,\cfrac{k_i}{k_i-1}\Big)\Big[1+\log^p \,\frac{k_s}{k_1-1}\Big]
		\\
		&\le  \
		2^{p-1}\Big[1+\log^p \,\frac{k_s}{k_1-1} \Big]+2^{p-1}\Big(\log\,\cfrac{k_s}{k_1-1}\Big)\Big[1+\log^p \,\frac{k_s}{k_1-1}\Big]
		\\
		& =  \
		2^{p-1}\Big[1+\log \,\frac{k_s}{k_1-1} +
		\log^p\, \frac{k_s}{k_1-1}+\log^{p+1} \,\frac{k_s}{k_1-1} \Big]
		\\
		&\le  \ 2^{p}\Big[1+ \log^{p+1} \,\frac{k_s}{k_1-1}\Big],
	\end{align*}
	which concludes the proof of Lemma \ref{AA}.
\end{proof}


  Now,
define $Q_{p, 1}, Q_{p, 2}: 
\ell^2\big(\mathbb{C}^{2^{p-1}}\big)\to \ell^2\big(\mathbb{C}^{2^{p-1}}\oplus\mathbb{C}^{2^{p-1}}\big)$
by
\[
Q_{p, 1}:   \{a_i\} \mapsto      
\Big\{\begin{bmatrix}
	a_i \\ 0
\end{bmatrix}\Big\}, \hspace{0.3cm } 
Q_{p, 2}:   \{a_i\} \mapsto      \Big\{\begin{bmatrix}
	0 \\ a_i
\end{bmatrix}\Big\}.
\]
Then 
$Q_{p, 1}^*, Q_{p, 2}^*:  
\ell^2\big(\mathbb{C}^{2^{p-1}}\oplus \mathbb{C}^{2^{p-1}}\big)\to \ell^2\big(\mathbb{C}^{2^{p-1}}\big)$ 
 act as 
\[
	Q_{p, 1}^*: 
\Big\{\begin{bmatrix}
	a_i \\ b_i
\end{bmatrix}\Big\} \mapsto \{a_i\},\hspace{0.3cm } 
Q_{p, 2}^*: 
\Big\{\begin{bmatrix}
	a_i \\ b_i
\end{bmatrix}\Big\} \mapsto \{b_i\}. 
\]
Further, for any $n\ge 1$ and $0\le p\le m$,   define $W_{p+1, n}(c):\ell^2\big(\mathbb{C}^{2^p}\big)\to\ell^2\big(\mathbb{C}^{2^p}\big)$  
by 
\begin{equation} \label{backW}
W_{p+1, n}(c)(x_1, x_2,  \dots) = 
\big(B_{p, n, 1}(c)x_{n+1}, B_{p, n, 2}(c)x_{n+2}, \dots\big).     
\end{equation}
Then~\eqref{variemaiii} yields
%
%
\begin{equation}\label{eq:T-power-n}
	T^n_{m, c} 
	=\ Q_{m, 1}T^n_{m-1, c}Q_{m, 1}^*+Q_{m, 2}T^n_{m-1, c}Q_{m, 2}^*+c   Q_{m, 1}W_{m, n}(c)Q_{m, 2}^*.
\end{equation}

\begin{lemma} \label{fuller}
$T_{m, c}$ is absolutely Ces\`aro bounded  for any  $c\ge 0.$ Also, for fixed $\epsilon>0$ and $m$, 
$T_{m, c}$ has absolute Ces\'aro constant at most $1+\epsilon$ for any $c$ sufficiently small.
\end{lemma}
 \begin{proof} 
For $c=0$, the first assertion is trivial. Fix $\epsilon>0$ and let $c>0.$   
Observe that Lemma \ref{AA} yields, for any $0\le p \le m$ and $0< c\le 1,$ 
\begin{align}
\|B_{p, n, j+1}(c)\|&\le  \sum\limits_{k=2+j}^{n+j+1}\log\Big(\frac{k}{k-1}\Big)2^{p-1}
\Big[1+\log^p  \Big(\frac{n+j+1}{j+1}\Big)\Big]         \notag \\
&\le 2^p\Big[1+\log^{p+1} \Big(\frac{n+j+1}{j+1}\Big) \Big].     \hspace {0.7 cm} \forall n\ge 1, j\ge 0. \label{coeffester}
\end{align}
%

 \noindent 
For a multi-index $I=(i_1, \dots, i_{k-1})\in \{1,2\}^{k-1}$ and $i_k\in \{1,2\}$, we set 
\[
Q_{m, I,i_k}=Q_{m, (i_1, \dots, i_{k})}:=Q_{m, i_1} Q_{m-1, i_2}\cdots Q_{m-k+1, i_k}. 
\] 
%
%
Applying \eqref{eq:T-power-n}, we can write, inductively, for any $n\ge 1$ and $c>0,$
\begin{align}\label{decompult}
T^n_{m, c} = &\ Q_{m, 1}T^n_{m-1, c}Q^*_{m, 1}+Q_{m, 2}T^n_{m-1, c}Q^*_{m, 2}+c   Q_{m, 1}W_{m, n}(c)Q^*_{m, 2}  \notag      \\ 
=&\sum_{i_1, i_2\in \{1, 2\}} Q_{m, (i_1,  i_2)}T^n_{m-2, c}Q^*_{m, (i_1, i_2)} \notag    \\ &\qquad +c\sum_{i_1\in\{1, 2\}}Q_{m, i_1, 1}W_{m-1, n}(c)Q^*_{m, i_1, 2} +c   Q_{m, 1}W_{m, n}(c)Q^*_{m, 2}  \notag   \\
=&\ \dots \notag   \\
=&\sum_{I\in\{1, 2\}^m}Q_{m, I}T^n_{0, c} Q^*_{m, I}
+c\sum_{\ell=1}^{m-1}\sum_{I\in\{1, 2\}^\ell}Q_{m, I,1} W_{m-\ell,n}(c) Q^*_{m, I,2}  +c   Q_{m, 1}W_{m, n}(c)Q^*_{m, 2}   \notag   \\ 
=&S^n
+c\sum_{\ell=1}^{m-1}\sum_{I\in\{1, 2\}^\ell}Q_{m, I, 1} W_{m-\ell,n}(c) Q^*_{m, I, 2}  +c   Q_{m, 1}W_{m, n}(c)Q^*_{m, 2}\, ,   
\end{align}
%
%
%
where $S$ is simply the usual (unweighted) backwards shift on $\ell^2(\mathbb{C}^{2^m})$ (and does not depend on $c$). In view of \eqref{decompult} and since $S$ has absolute Ces\'aro constant equal to $1,$ to show that $T_{m,c}$ has absolute Ces\'aro constant at most $1+\epsilon$ for $c>0$ sufficiently small, it suffices to find $K=K(m)>0$  such that 
\begin{equation} \label{auxiliary}
\sup_{N\in\mathbb{N}}\frac{1}{N}\sum_{n=1}^N\|W_{p+1, n}(c)b\|\le K\|b\|,
\end{equation}
for any $0\le p< m, 0<c\le 1$ and $b\in\ell^2(\mathbb{C}^{2^{p}})$.
But this follows from \eqref{mullereq} and \eqref{coeffester}. Indeed, taking e.g. $a=1/4$ in the former yields, for any $0\le p< m, 0<c\le 1$,  
%
%
\begin{align}
 & \sup_{N\in\mathbb{N}}\frac{1}{N}\sum_{n=1}^N\|W_{p+1, n}(c)b\|                       \notag \\ 
 =&\sup_{N\in\mathbb{N}}\frac{1}{N}
   \sum_{n=1}^{N}\sqrt{\sum_{j=1} ^{\infty}\Big\|B_{p, n, j}b_{n+j}\Big\|^2 }            \notag  \\  
\le&\sup_{N\in\mathbb{N}}\frac{1}{N}\sum_{n=1}^{N}\sqrt{\sum_{j=1} ^{\infty}  2^{2p}\Big[1+\log^{p+1} \frac{n+j}{j}\Big]^2     \|b_{n+j}\|^2} \notag \\
\le&\sup_{N\in\mathbb{N}}\frac{1}{N}\sum_{n=1}^{N}\sqrt{\sum_{j=1} ^{\infty}  2^{2p+1}\Big[1+ \log^{2(p+1)} \frac{n+j}{j} \Big]    \|b_{n+j}\|^2}  \notag \\
\le&\sup_{N\in\mathbb{N}}\frac{1}{N}\sum_{n=1}^{N}\sqrt{\sum_{j=1} ^{\infty}  2^{2p+1}  \|b_{n+j}\|^2}
\,+ \,  \sup_{N\in\mathbb{N}}\frac{1}{N}\sum_{n=1}^{N}\sqrt{\sum_{j=1} ^{\infty}  2^{2p+1} \Big(\log^{2(p+1)} \frac{n+j}{j}    \Big) \;\|b_{n+j}\|^2}                           \notag \\ 
\le &\ 
2^{p+1/2}\|b\|+
\sup_{N\in\mathbb{N}}\frac{2^{p+1/2}}{N}\sum_{n=1}^{N}\sqrt{\sum_{j=1} ^{\infty}  (4p+4)^{2(p+1)}
	\Big(\cfrac{n+j}{j}\Big)^{1/2}\,
	 \|b_{n+j}\|^2} 
          \notag        \\ 
\le &\
2^{p+1/2}\|b\|+
\sup_{N\in\mathbb{N}}\frac{2^{3p+5/2}(p+1)^{p+1}}{N}\sum_{n=1}^{N}\sqrt{\sum_{j=1} ^{\infty}  \Big(\cfrac{n+j}{j}\Big)^{1/2}    \|b_{n+j}\|^2} 
    \notag   \\ 
\le &\ 
2^{p+1/2}\|b\|+2^{3p+5/2}(p+1)^{p+1} \sqrt{6}\|b\| 
    \notag \\ 
\le  &\ (2^{m-1/2}+2^{3m-1/2}m^m\sqrt{6})\|b\|.\label{variemai}
\end{align}
Here we have applied the inequality 
\[
\log^{2(p+1)}\,\cfrac{n+j}j
\le (4p+4)^{2(p+1)}\Big(\cfrac{n+j}j\Big)^{1/2}, 
\]
which follows from an easy inequality $\log x\le x$ ($x>0$) by putting $x=\Big(\cfrac{n+j}j\Big)^{1/(4p+4)}$. 
Thus, \eqref{auxiliary} holds with \[
K=2^{m-1/2}+2^{3m-1/2}m^m\sqrt{6}.
\]   
Finally, observe that in \eqref{decompult}, 
\[
\big\|\sum_{I\in\{1, 2\}^\ell} 
Q_{m, I,1} W_{m-\ell,n}(c) Q^*_{m, I,2}b \big\|
\le \|W_{m-\ell,n}(c)b\|
 \]
for all $\ell$ and $b$.  
Hence, choosing $c\le \epsilon /(Km)$ and employing \eqref{decompult}-\eqref{auxiliary} allows us to write 
\begin{align*}
 \frac{1}{N}\sum_{n=1}^N\|T^n_{m, c}x\| 
& \le \  \frac{1}{N}\sum_{n=1}^N\|S^nx\|+\frac{c}{N}\sum_{n=1}^N\sum_{\ell=0}^m\|W_{m-\ell, n}(c) x\| \\
&\le  \ \|x\|+mcK\|x\| 
 \le  \  (1+\epsilon)\|x\|,
\end{align*}
and so $T_{m,c}$ has absolute Ces\'aro constant at most $1+\epsilon$ in this case. 
This concludes the proof of Lemma \ref{fuller}.
 \end{proof}

We are now ready to prove Theorem \ref{cesaromain}.
%
%
\begin{proof}[Proof of Theorem \ref{cesaromain}]
Although statement (ii) is stronger than (i), we begin with~(i). 
\medskip 

\noindent \underline{Proof of part (i).}  
Fix $m\ge 1$ and let $\epsilon>0.$ 
First, we recall from \eqref{variemaiii} that the $2^m\times 2^m$ matrix $\prod_{i=2}^NA_{m, i}(c)$, when viewed as a $(2\times 2)$-block matrix acting on $\ell^2\big(2^{m-1}\big)\oplus \ell^2\big(2^{m-1}\big)$, has $(1, 2)$ entry equal to 
\[
cB_{m-1, N-1, 1}(c)
=c\sum_{i=2}^N\Big(\log\frac{i}{i-1}\Big)\prod_{2\le j\neq i}^NA_{m-1, j}(c).\]
This last matrix can, in turn, be viewed as  a $(2\times 2)$-block matrix acting on $\ell^2\big(2^{m-2}\big)\oplus \ell^2\big(2^{m-2}\big)$, in which case its $(1, 2)$ entry (and thus the $(1, 2^2)$ entry of $\prod_{i=2}^NA_{m, i}(c)$) is given by
\[c^2\sum_{i=2}^N\log\bigg(\frac{i}{i-1}\bigg)\sum_{2\le j\neq i}^N\Big(\log\frac{i}{i-1}\Big)\prod_{2\le k\neq i, j}^NA_{m-2, k}(c).\]
Continuing inductively, we obtain that the $(1, 2^m)$ entry of 
$\prod_{i=2}^NA_{m, i}(c)$ is given by 
	\begin{align} 
	& c^m\sum_{i_1=2}^N\Big(\log\frac{i_1}{i_1-1}\Big)\sum_{2\le i_2\neq i_1}^N\Big(\log\frac{i_2}{i_2-1}\Big)\cdots \sum_{\substack{2\le i_{m}, \\ i_{m}\neq i_1, \dots, i_{m-1}}}^N \Big(\log\frac{i_{m}}{i_{m}-1}\Big) \nonumber     \\  
	 \ge \  & c^m \log^m \,\frac{N}{m}\, . \label{power} 
	 \end{align}
Here we have used that, independently of $i_1,\dots, i_{m-1}$, 
the last 
sum is greater than or equal to $\log\frac Nm$.  One then applies the same estimate successively to 
the rest of the sums, going from right to left.

Now, in view of the proof of Lemma \ref{fuller}, we can choose 
\begin{equation}\label{frla-c}
c=c(\epsilon,m)=\cfrac{\epsilon}{m\big(2^{m-1/2}+2^{3m-1/2}m^m\sqrt{6}\big)}, 
\end{equation}
so that $T_{m, c}$ has absolute Ces\'aro constant at most $1+\epsilon.$ Further, let $\{\eta_i\}$ and $\{e_i\}$ denote the usual bases of $\ell^2$ and $\mathbb{C}^{2^m}$, respectively. For $N$ fixed,  define $P=P_N: \ell^2\big(\mathbb{C}^{2^m}\big)\to \ell^2\big(\mathbb{C}^{2^m}\big)$ to be the projection onto $V=V_N:=\text{span}_{1\le i\le N}\{\eta_i\}\otimes \mathbb{C}^{2^m}$.  Consider the operator 
\begin{equation} \label{xdef}
   X_{m, N}(c):=PT_{m, c}\big|_V=T_{m, c}\big|_V. 
\end{equation}
Applying $X^{N-1}_{m, N}(c)$ to $\eta_N\otimes e_{2^m} $ yields, in view of \eqref{power},
\begin{align*}
  \|X^{N-1}_{m, N}(c)(\eta_N\otimes e_{2^m})\|&=\|T^{N-1}_{m, c}(\eta_N\otimes e_{2^m})\| \\
  &=\Big\|\eta_1\otimes \bigg(\prod_{i=2}^{N}A_{m, i}(c)\bigg)e_{2^m}\Big\| \\
  &\ge  c^m \log^m\bigg(\frac{N}{m}\bigg).
\end{align*}
Thus, $X_{m, N}(c)$ is an operator acting on the $N'=2^mN$-dimensional space $V_N$ and such that, for 
$N'\ge m^22^{2m},$
%
%
%
%
%
%
\begin{equation}\label{penult}
	\big\|X_{m, N}^{N-1}(c)\big\|\ge c^m\log^m\frac{N}{m} 
	=c^m\log^m\,\frac{N'}{m2^m} 
	\ge \cfrac{c^m}{2^m} \,\log^m N'= C(\eps,m)\,\log^m N',    	
\end{equation}
where $C(\eps, m)=2^{-m}c(\eps, m)^m$. 
Finally,  since 
$PT_{m, c}\big|_V=T_{m, c}\big|_V$, 
we obtain, for any $x\in V$ and $n\ge 1,$
\[\big\|X_{m, N}^n(c)x\big\|=\|T^n_{m, c}x\|, \]
which yields that $X_{m, N}(c)$ has absolute Ces\'aro constant at most $1+\epsilon$ as well. Combining this observation with \eqref{penult}, we obtain that  \begin{equation}\label{P-AC-log-m}
P_{AC}(N', 1+\epsilon)\ge C(\epsilon, m) \log^m (N'),
\end{equation} 
for all 
$N'\ge m^2 2^{2m}$ such that $N'$ is also a multiple of $2^m$. Finally, we will show that
\begin{equation}\label{P-AC-log-m-redux}
	P_{AC}(N'', 1+\epsilon)\ge 2^{-m}C(\epsilon, m) \log^m (N''),
\end{equation} 
for all $N''\ge m^2 2^{2m}$.
Indeed, let $N''$ satisfy $N''\ge m^2 2^{2m},$ with $N'$ the largest integer such that $N'\le N''$ and $N'$ is a multiple of $2^m$. Since $P_{AC}(N'', K)$ is increasing with respect to $N''$ (assuming $K$ fixed), \eqref{P-AC-log-m} allows us to write
\begin{align*}
 P_{AC}(N'', 1+\eps)&\ge P_{AC}(N', 1+\eps) \\
 &\ge  C(\eps, m) \log^m (N') \\
 &\ge  C(\eps, m) \log^m (N''/2) \\ 
 &\ge  2^{-m}C(\eps, m) \log^m (N''),
\end{align*}
as desired. \eqref{P-AC-log-m-redux} is now easily seen to imply assertion (i).  \\

%
\noindent \underline{Proof of part (ii).}
First we notice that 
\begin{equation}\label{C-eps-mew}
2^{-m}C(\eps, m)=\frac{c(\eps, m)^m}{4^m}
=\frac{\eps^m}{\kappa_m},  
\end{equation}
where 
\[
\kappa_m=
(4m)^m[2^{m-1/2}+2^{3m-1/2}m^m\sqrt{6}]^m. 
\]
Clearly,
\begin{equation}\label{est-kappa}
\kappa_m\le (4m)^m\,(2\cdot 2^{3m}m^m)^m
=2^{3m^2+3m}m^{m^2+m}=(8m)^{m^2+m}.  
\end{equation}
Fix some $N\in \N$ and $\eps\in(0, \eps_0)$.  
Then, by \eqref{P-AC-log-m-redux} and \eqref{C-eps-mew}, 
\[
\log P_{AC}(N, 1+\epsilon)\ge m\log( \eps \log N)
-(m^2+m)\log(8m). 
\]
 Set $m=\lfloor \frac v{2\log v}\rfloor$, where $v=\log(\eps \log N)$. This choice of $m$ is, in a sense, close to the value that maximizes this last lower bound. Let us check that $N\ge m^22^{2m}$, which 
will allow us to apply~\eqref{P-AC-log-m-redux}.

%
%
 
We are assuming that $\eps \log N\ge M$, where $M$ is sufficiently large so that
\[
m^22^{2m} < e^{2m} \le \exp \frac v{\log v}. 
\]
Denoting by $\log_{(k)}$ the $k$th iterated logarithm, we observe that 
\[
\frac v{\log v}\le \frac{\log_{(2)} N}{\log_{(3)} N}\le \log N, 
\]
because $v\in [\log M,\log_{(2)} N]$ and the function $\frac x{\log x}$ increases on this interval, if $M$ is assumed sufficiently large. 
Combining the above inequalities, we conclude that $m^22^{2m}\le N$ whenever $\eps \log N\ge M$ and $M$ is sufficiently large.  

Thus, we may apply ~\eqref{P-AC-log-m-redux}.  Choose a constant $\alpha\in(0,\frac 14)$; in view of ~\eqref{C-eps-mew} 
and~\eqref{est-kappa}, we get 
\begin{align*}
\log P_{AC}(N, 1+\epsilon)
&\ge m\log(\eps \log N)-(m^2+m)\log(8m)        \notag\\
&\ge \Big(\frac v{2\log v}-1\Big)v-
\Big(
\frac{v^2}{4\log^2v}+\frac{v}{2\log v}
\Big)\,\log v                          \notag    \\
& \ge \frac{\alpha v^2}{\log v}         \notag         \\
& =\cfrac{\alpha\log^2\log N^{\eps}}{\log_{(3)}N^{\eps}},
\end{align*}
if $v\ge \log M$   and $M$ is sufficiently large. 
This yields assertion (ii)  of Theorem~\ref{cesaromain}.
\end{proof}
%
%
%
 
\subsection{Lower bounds for Kreiss conditions} \label{kreissssubsection} Our work in subsection \ref{cesarosubsection} can now be adapted to yield various Kreiss-type bounds.  
\begin{theorem} \label{uniformthm}
  Fix $\epsilon>0$ and $m\in\mathbb{N}$. Then, there exists a constant $C=C(\epsilon, m)>0$ such that for all $N\ge 1,$
\begin{equation} \label{mainU}
    P_U(N, 1+\epsilon)\ge C \log^m (N).
\end{equation}
\end{theorem}
 \noindent The theorem will be a straightforward consequence of the following lemma, which shows that the residual $W_{p, k}(c)$ terms in \eqref{decompult} satisfy a uniform Kreiss-like condition. 
\begin{lemma}\label{anotherprelimyes}
Let $m\ge0$. Then, there exists $C_0>0$ such that, for any $0\le p\le m-1$ and $0\le c\le 1,$
\begin{equation}\label{uprelim}
\bigg\|\sum_{k=0}^Nz^{-k-1}W_{p+1, k}(c)\bigg\|\le \cfrac{C_0}{|z|-1},    
\end{equation}    
for all $N\ge 1$ and $|z|>1$.
\end{lemma}
\begin{proof}
Fix $m\in\mathbb{N}$ and $0\le p\le m-1.$ Recall the matrices $A_{p, k}(c),$ defined in \eqref{adef}. For any $0\le c\le 1$ and $k\ge 2,$ we further define \[\widetilde{A}_{p+1, k}(c)=\begin{bmatrix}
A_{p, k}(c) & \log \big(\frac{k}{k-1}\big)I_{2^p} \\
0 & A_{p, k}(c)
\end{bmatrix}\]
and  $\widetilde{T}_{p+1, c}: \ell^2\big(\mathbb{C}^{2^{p+1}}\big)\to \ell^2\big(\mathbb{C}^{2^{p+1}})$ to be the backwards weighted shift with weights $\big\{\widetilde{A}_{p+1, k}(c)\big\}_{k\ge 2}$.
It is easily seen that, for any $n\ge 1,$
\begin{equation}
\widetilde{T}^n_{p+1, c} 
=\ Q_{p+1, 1}T^n_{p, c}Q^*_{p+1, 1}+Q_{p+1, 2}T^n_{p, c}Q^*_{p+1, 2}+   Q_{p+1, 1}W_{p+1, n}(c)Q^*_{p+1, 2}. \label{decomplessult}
\end{equation}
In view of the proof of Lemma \ref{fuller}, there exists a constant $C_1>0$ only depending on $m$ such that $T_{0,c}, T_{1, c},\dots, T_{p, c}$ are absolutely Ces\'aro bounded with constant at most $C_1$, for any $0\le c\le 1.$
Combining with \eqref{decomplessult} and \eqref{variemai}, we get the existence of another constant $C_2>0$, again depending only on $m,$ such that $\widetilde{T}_{p+1, c}$ is absolutely Ces\'aro bounded with constant at most $C_2,$ for any $0\le c\le 1.$\par  Now, it is well-known that the uniform Kreiss constant of any operator is at most an (absolute) numerical constant times its absolute Ces\'aro constant (see e.g. the proofs of \cite[Corollary 3.2]{Mont} and \cite[Theorem 2]{Quaestiones}). In particular, we can find a constant $C>0$ only depending on $m$ such that 
$T_{0,c}, T_{1, c},\dots, T_{p, c}$ and $\widetilde{T}_{p+1, c}$ all satisfy the uniform Kreiss condition with constant at most $C$, for any $0\le c\le 1.$ As \eqref{decomplessult} yields
\begin{align*}
&\sum_{k=0}^Nz^{-k-1}\widetilde{T}^k_{p+1, c} \notag   \\
 = &    \sum_{k=0}^Nz^{-k-1}Q_{p+1, 1}T^k_{p, c}Q^*_{p+1, 1}  \notag \\ 
 & +\sum_{k=0}^Nz^{-k-1}Q_{p+1, 2}T^k_{p, c}Q^*_{p+1, 2}+\sum_{k=0}^Nz^{-k-1}Q_{p+1, 1}W_{p+1, k}(c)Q^*_{p+1, 2},
\end{align*}
for any $N\ge1 , |z|>1$, we may conclude that \eqref{uprelim} holds with $C_0=3C,$ as desired.
\end{proof}

 \begin{proof}[Proof of Theorem \ref{uniformthm}]
Fix $m\ge 1$ and $\epsilon>0$ and define $T_{m, c}$ as in Subsection \ref{cesarosubsection}, for any $c>0.$ In view of \eqref{decompult} and \eqref{uprelim} (and the basic fact that the unweighted backwards shift $S$ has uniform Kreiss constant equal to $1$), we deduce that, if $c$ is chosen sufficiently small, $T_{m, c}$ will have uniform Kreiss constant at most $1+\epsilon.$ Now, for any $N>m$, define $P=P_N$ and $X_{m, N}(c)$ as in \eqref{xdef}. Since $PT_{m, c}P=T_{m, c}P,$ we have 
\[\sum_{k=0}^Nz^{-k-1}\big(X_{m, N}(c)\big)^k=P\bigg(\sum_{k=0}^Nz^{-k-1}T_{m, c}^k\bigg)P ,\]
for any $N\ge 1$ and $|z|>1, $ and so $X_{m, N}(c)$ has uniform Kreiss constant at most $1+\epsilon$ as well. \eqref{mainU} now follows readily from \eqref{penult}.
\end{proof}

\begin{remark}
Since $P \ge P_{U},$ \eqref{main} follows immediately from Theorem \ref{uniformthm}.
\end{remark}

\begin{remark}\label{poweronly}
Consider the direct sum
\[Y_{m, c}:=\oplus_{N}X_{m, N}(c).\]
Since direct sums respect (uniform) Kreiss constants, we obtain that, for $c$ sufficiently small, $Y_{m, c}$ has uniform Kreiss constant at most $1+\epsilon$ and satisfies 
\[
\|Y^n_{m, c}\|\gtrsim \log^m (n). 
\]
\end{remark}
\begin{remark} \label{further2}
Using an analogous line of reasoning, it is possible to obtain a lower bound for $P_S$ as well. Indeed, we first recall that one of the key ingredients in the proofs of Theorems \ref{cesaromain} and \ref{uniformthm} is [\cite{AbsolCesaromuller},Theorem 2.1], which  asserts that the weighted backwards shift on $\ell^2$ with weights $w_k=\big(\frac{k}{k-1}\big)^a, a>0,$ is absolutely Ces\'aro bounded. To treat the case of the strong Kreiss condition, one instead turns to \cite[Proposition 4.9]{CohenCunyEisnerLin}; this result shows that, if we choose $w_k=[\log(k+2)/\log(k+1)]^a$, then the associated weighted backwards shift satisfies the strong Kreiss condition, for any $a>0.$ Now, given $m\ge 1$ and $\epsilon>0,$ one can define inductively, for any $c>0$ and $k>2$, $\widetilde{A}_{0, k}(c)=1$ and 
\[\widetilde{A}_{p+1, k}(c) = 
\begin{bmatrix}
\widetilde{A}_{p, k}(c) & c \log\big(\frac{\log(k+2)}{\log 2}\big) I_{2^p} \\
0 & \widetilde{A}_{p, k}(c)
\end{bmatrix}.\]
The proofs of Theorems \ref{cesaromain} and \ref{mainU} can then be repeated mutatis mutandis (with \eqref{mullereq} replaced by the analogous inequality implied by \cite[Proposition 4.9]{CohenCunyEisnerLin}) to yield the existence of a constant $C_S=C_S(\epsilon, m)$ such that
\begin{equation} \label{strongkreissineq}
 P_{S}(N, 1+\epsilon)\ge C_S  \log^m (\log N), \hspace{0.3 cm} \forall N\ge 1.   
\end{equation}
Further, by considering direct sums as in Remark \ref{poweronly}, we may also obtain an operator $V_{m, c}$ that, for $c$ sufficiently small, has strong Kreiss constant at most $1+\epsilon$ and also
satisfies 
\[
\|V_{m, c}^n\|
\gtrsim \log^m(\log n).
\]
This improves McCarthy's growth rate of $\log^{1/2}(\log n)$  \cite[Example 2]{McCarthy71lost}. We omit the details.
\end{remark}

\section{Proof of Theorem \ref{genKreiss2}} 
\label{genKreiss2sec}

\begin{proof}[Proof of Theorem \ref{genKreiss2}]
First we observe that for any v-type curve 
$\ga(\theta)=r(\theta)e^{-i\theta}$, one has 
\begin{equation}\label{limsup=0}
	\limsup\limits_{\theta \to 0^\pm}\cfrac{|\theta|}{r(\theta)-1}=0. 
\end{equation}
Indeed, if it were not true, say, for 
$\limsup_{\theta \to 0^+}$, then there would exist $\tau$ such that 
$\theta/(r(\theta)-1)>\tau>0$ on a sequence of positive angles $\theta_n$ such that 
$\theta_n\to 0$. Since, by property (a) of Definition \ref{v-type-def}, $r(\cdot)$ increases on $[0,\de_+]$, this implies 
that 
$1/(r(\theta)-1)>\tau/\theta_n$ for 
$\theta\in [\theta_n/2,\theta_n]$, $n=1,2\dots$, which contradicts the finiteness of the integral from property (c). 

%
%
%

Now, by Paulsen's theorem \cite{Paulsensimilaritythm}, it suffices to show that $\overline{\DD}$ is a complete $K$-spectral set for $T$ for some $K>0.$ We will only show $K$-spectrality; the proof generalizes to the complete case by tensoring everything appropriately.  Alternatively, we can apply a result by  deLaubenfels \cite{deLauben}, which says that a Kreiss operator with finite peripheral spectrum is polynomially bounded if and only if it is completely polynomially bounded. (The same result for Ritt operators had been previously proved by Le Merdy \cite[Theorem 5.1]{LMsimprob}). 
\par

%
%
%
Continue the function $r(\cdot)$ to $[-\pi,\pi]$ by adding constant segments to its graph, so that it increases on $[0,\pi]$, decreases on $[-\pi,0]$ and satisfies $r(-\pi)=r(\pi)=r(\delta_+)=r(-\delta_-)$. Without loss of generality, we may assume that $\varepsilon$ is defined on the interval $(0, \nu]$, where
$1+\nu=r(-\pi)=r(\pi)$ (note $\nu>0$). 
%
%
Extend $\gamma(\theta)=r(\theta)e^{-i\theta}$ 
to all $\theta\in [-\pi,\pi]$. Then, $\gamma$ is a Jordan curve that contains $\DD$ in its interior and intersects $\T$ at $\{1\}$. Also, let $\{C_n\}_{n\ge 1}$ be a family of circles centered at the origin with radii $\{r_n\}_{n\ge 1}$ and such that 
\begin{equation} \label{r1}
    r_1<1+\nu-\nu/(1+\varepsilon(\nu))
\end{equation} 
and $r _n\searrow 1$. 
Further, for any $n\ge 1$, let $0<a_n<\pi$ and $-\pi<b_n<0$ denote the angles with
\begin{equation} \label{anbn}
r(a_n)-\frac{r(a_n)-1}{1+\varepsilon(r(a_n)-1)}
=r(b_n)-\frac{r(b_n)-1}{1+\varepsilon(r(b_n)-1)}
=r_n.
\end{equation}
The existence of these angles is guaranteed by the conditions $r(\theta)\to 1$ as $\theta\to 0$ and 
\[
r_n<r_1<1+\nu-\frac{\nu}{1+\varepsilon(\nu)}. 
\] 
Moreover, we assume that $a_n$ is the largest solution of \eqref{anbn} on $[0,\pi]$ and that 
$b_n$ is the smallest solution of \eqref{anbn} on $[-\pi,0].$ 
%
%
%
Our goal is to find a constant $K>0$ (not depending on $n$) so that 
\[\|p(T)\|\le K\max_{z\in C_n}|p(z)|, \hspace{0.6 cm} \forall n\ge 1,\hspace{0.1 cm} p\in \mathbb{C}[x].\]
Letting $n\to\infty$ will then yield the desired result. \par 

Now, let 
%
%
%
\[
\mu(\sigma_n(\theta), T)=\frac{1}{2\pi i}\bigg(\frac{\sigma'_n(\theta)}{\sigma_n(\theta)-T} -\frac{\overline{\sigma'_n(\theta)}}{\overline{\sigma_n(\theta)}-T^{\ast}} \bigg) 
\]
%
%
denote the double-layer potential operator, where $\sigma_n(\theta)=r_ne^{-i\theta} $ and $\sigma'_n(\theta)=-ir_ne^{-i\theta}$, with $\theta\in (-\pi, \pi]$.  The main ingredient of the proof will be the following assertion. We use $A\succeq 0$ to denote that an operator $A$ is positive semi-definite.


%
%
%
%
%
%

\begin{lemma}\label{lem:Pn-Dn}
	There exist piecewise continuous functions 
	$R_n:[-\pi,\pi]\to (0,+\infty)$
	and $\mathcal{B}(H)$-valued functions 
	$P_n(\tht), \fD_n(\tht)$ such that 
	
	\begin{itemize}
		\item[(a)]
		$\displaystyle 
		\mu(\sigma_n(\theta), T)+\frac{r_n}{2\pi}\, R_n(\tht)^{-1} =P_n(\tht)+\fD_n(\tht)$ 
		($n\in \N, \; \tht\in [-\pi,\pi])$; 

%
%

		\item[(b)]
		$P_n(\tht)\succeq  0$ for all $n$ and $\tht$; 
		
		\item[(c)]
		$\limsup_{n\to\infty}\int_{-\pi}^\pi R_n(\tht)^{-1}\is d\tht<+\infty$; 
		
		\item[(d)]
		$\limsup_{n\to\infty}\int_{-\pi}^\pi \|\fD_n(\tht)\|\is d\theta<+\infty$. 
		
	\end{itemize}
\end{lemma}

\noindent Roughly, the idea is that $\mu(\sigma_n(\theta), T)$ can be made positive semi-definite up to an error term $\fD_n(\tht)$ by adding a function $R_n(\theta)^{-1}$ such that both $\|\fD_n(\tht)\|$ and $R_n(\theta)^{-1}$ are ``controllable" (in this case, uniformly integrable) as $n\to\infty. $
The functions $R_n(\tht),P_n(\tht), \fD_n(\tht)$ will be given later by explicit formulas. The proof is postponed; we will first show how the lemma can be used to complete the proof of Theorem \ref{genKreiss2}.

Let $p$ be a polynomial. 
We will estimate $\|p(T)\|$. Put 
\[
\la_n(\tht):=\frac{r_n}{2\pi}\,R_n(\tht)^{-1}.
\]
Letting $\cC$ denote the Cauchy transform, we have
\[
{ \cC(\overline{p})(T)} 
	:=\frac{1}{2\pi i}\int_{0}^{2\pi}\sigma_n'(\theta)(\sigma_n(\theta)-T)^{-1} \ \overline{p(\sigma_n(\theta))}d\theta=\overline{p(0)},\]
Thus, we obtain (writing $p_n(\theta)$ for $p(\sigma_n(\theta)$))
\begin{equation}\label{p(T)}
\begin{aligned}
	p(T)&=p(T)+[\cC(\overline{p})(T)]^*-[\cC(\overline{p})(T)]^* \\
	&=\int_{-\pi}^{\pi}\mu(\sigma_n(\theta), T)p_n(\theta) d\theta                -p(0)I \\
	&=\int_{-\pi}^{\pi}\big[\mu(\sigma_n(\theta), T)+\lambda_n(\theta)\big]p_n(\theta) d\theta-\int_{-\pi}^{\pi}\lambda_n(\theta)p_n(\theta)d\theta-p(0)I \\
	&=\int_{-\pi}^{\pi}P_n(\theta)p_n(\theta) d\theta+\int_{-\pi}^{\pi}\fD_n(\theta)p_n(\theta) d\theta-\int_{-\pi}^{\pi}\lambda_n(\theta)p_n(\theta)d\theta-p(0)I.
\end{aligned}
\end{equation}

Since $P_n(\theta)\succeq  0$, we have  
\begin{equation}\label{C-Schw}
\Big\|
\int_{-\pi}^{\pi}P_n(\theta)p_n(\theta) d\theta
\,\Big\|\le 
\sup_{z\in C_n}|p(z)|\cdot
\Big\|
\int_{-\pi}^{\pi}P_n(\theta) d\theta\,
\Big\|\; .
\end{equation}
Indeed, it suffices to consider the case when 
$|p|\le 1$ on $C_n$. Let 
$u,v$ be arbitrary unit vectors. 
Then by applying twice the  
Cauchy-Schwarz inequality we get that 
\begin{multline*}
\Big|
\Big\langle \Big(\int_{-\pi}^{\pi}P_n(\theta)p_n(\theta)d\theta\,\Big) u, v  \Big\rangle
\Big|             
\le 
\int_{-\pi}^{\pi}
\big| \langle P_n(\theta)u, v \rangle \big| 
\, d\tht    \\
 \le 
\int_{-\pi}^{\pi}
\big\langle
P_n^{1/2}(\tht)u, P_n^{1/2}(\tht)u
\big\rangle 
\, 
\big\langle
P_n^{1/2}(\tht)v, P_n^{1/2}(\tht)v
\big\rangle 
\; d\tht     
\le 
\Big\|
\int_{-\pi}^{\pi}P_n(\theta) d\theta\,
\Big\|\; , 
\end{multline*}
which gives \eqref{C-Schw}. 
Hence, crashing through with norms in \eqref{p(T)}, we obtain 
\begin{align*}
	\|p(T)\|&\le \bigg(\bigg\|\int_{-\pi}^{\pi}P_n(\theta)d\theta\bigg\|+\int_{-\pi}^{\pi}\|\fD_n(\theta)\|d\theta+\int_{-\pi}^{\pi}\lambda_n(\theta)d\theta+1\bigg)
	\sup_{z\in C_n}|p(z)|\\
	&\le \bigg(\bigg\|\int_{-\pi}^{\pi}\mu(\sigma_n(\theta), T)d\theta\bigg\|+2\int_{-\pi}^{\pi}\|\fD_n(\theta)\|d\theta+2\int_{-\pi}^{\pi} \lambda_n(\theta) d\theta+1\bigg)
	\sup_{z\in C_n}|p(z)| \\
	&= \bigg(2+2\int_{-\pi}^{\pi}\|\fD_n(\theta)\|d\theta+2\int_{-\pi}^{\pi}\lambda_n(\theta)d\theta+1\bigg)
	\sup_{z\in C_n}|p(z)| \\
	&= \bigg(3+2\int_{-\pi}^{\pi}\|\fD_n(\theta)\|d\theta+2\int_{-\pi}^{\pi}\lambda_n(\theta) d\theta\bigg)
	\sup_{z\in C_n}|p(z)| 
\end{align*} 
for every $n\ge 1.$ Thus, 
\begin{align*}
	\|p(T)\|&\le \limsup\limits_{n\to\infty}\bigg(3+2\int_{-\pi}^{\pi}\|\fD_n(\theta)\|d\theta+2\int_{-\pi}^{\pi}|\lambda_n(\theta)|d\theta\bigg)\sup_{z\in\mathbb{T}}|p(z)|\\
	&:=K\sup_{z\in\mathbb{T}}|p(z)|,
\end{align*}
where $K$ is finite because of Lemma ~\ref{lem:Pn-Dn}. 
This completes the proof of Theorem~\ref{genKreiss2} (modulo 
Lemma~\ref{lem:Pn-Dn}). 
\end{proof}

To prove Lemma~\ref{lem:Pn-Dn}, 
we require the following well-known Poisson-type  factorization. The proof is omitted (see e.g.  \cite[Lemma 7]{Crouzeix-Greenbaum}). 
\begin{lemma} \label{operpoisson}
Assume $v, v', C\in\mathbb{C}$ and $R>0$ are such that $v-C=iRv'$ and $|v'|=1$. Then, for any $T\in\mathcal{B}(H)$ with 
$v, C\notin\si(T)$, 
%
%
we have
\[\frac{1}{i}\big(v'(v-T)^{-1}-\overline{v'}\big(\overline{v}-T^*)^{-1}\big)+\frac{1}{R}I\]
\[=R(v-T)^{-1}(T-C)\Big(\frac{1}{R^2}-(T-C)^{-1}(T-C)^{-\ast}\Big)(T-C)^{\ast}(v-T)^{-\ast}.\]
\end{lemma}
%
%
%
%
%
%
%

\begin{proof}[Proof of Lemma \ref{lem:Pn-Dn}]
For each $n$, we divide $[-\pi,\pi]$ into two sets:
\[
[-\pi,\pi]=I_{1n}\cup I_{2n}, \quad \text{where }
    I_{1n}=[-\pi,b_n]\cup [a_n,\pi], \; I_{2n}=[b_n,a_n]. 
\] 
The expressions for $R_n(\tht), P_n(\tht), \fD_n(\tht)$ will depend on whether 
$\tht$ belongs to $I_{1n}$, to $[b_n,0]$ or to $[0,a_n]$. 

\medskip 
{\bf Case 1:} $\tht\in I_{1n}=[-\pi, b_n]\cup [a_n, \pi]$. 
Applying Lemma \ref{operpoisson} with $v=r_n e^{-i\theta}, v'=-ie^{-i\theta}, C=r(\tht)e^{-i\theta}, R=r(\theta)-r_n$, we obtain (recall that $\sigma_n(\theta)=r_ne^{-i\theta}$ and $\sigma'_n(\theta)=-ir_ne^{-i\theta}$)
\begin{equation}\label{1stint}
\begin{aligned}
{ \mu}
&
{ (\sigma_n(\theta), T)+\frac 1{2\pi}\,\frac{r_n}{r(\theta)-r_n}
}                                     \\
=&\frac 1{2\pi}\,r_n(r(\theta)-r_n)A_1^{-1}B_1\Big(\frac{1}{(r(\theta)-r_n)^2}-\big(T-r(\theta)e^{-i\theta}\big)^{-1}\big(T-r(\theta)e^{-i\theta}\big)^{-\ast}\Big)B_1^*A_1^{-\ast},  
\end{aligned}
\end{equation}
where $A_1=r_ne^{-i\theta}-T, B_1=T-r(\theta)e^{-i\theta}$. 
Set $R_n(\tht)=r(\tht)-r_n$ for 
$\tht\in I_{1n}=[-\pi, b_n]\cup [a_n, \pi]$. 
Since, by assumption, we have
\[\big(T-r(\tht)e^{-i\theta}\big)^{-1}\big(T-r(\tht)e^{-i\theta}\big)^{-\ast}\le \frac{[1+\varepsilon(r(\tht)-1)]^2}{(r(\tht)-1)^2},\]
we rewrite~\eqref{1stint} as 
\begin{align*}
&
{ \mu(\sigma_n(\theta), T)+ \frac{r_n}{2\pi}\,R_n(\tht)^{-1}
}     \\
=&\ \frac{r_n}{2\pi}\,(r(\theta)-r_n)A_1^{-1}B_1\bigg(\frac{\big[1+\varepsilon(r(\theta)-1)\big]^2}{(r(\theta)-1)^2}-\big(T-r(\theta)e^{-i\theta}\big)^{-1}\big(T-r(\theta)e^{-i\theta}\big)^{-\ast}\bigg)B_1^*A_1^{-\ast}+\fD_n(\theta) \\
=&P_n(\theta)+\fD_n(\theta),
\end{align*}
where $P_n(\theta)\succeq 0$  on $I_{1n}$ and 
\begin{equation} \label{dif2}
\fD_n(\theta)=\frac{r_n}{2\pi}\,(r(\theta)-r_n)A_1^{-1}B_1\bigg(\frac{1}{(r(\theta)-r_n)^2}-\frac{\big[1+\varepsilon(r(\theta)-1)\big]^2}{(r(\theta)-1)^2}\bigg)B_1^*A_1^{-\ast}.    
\end{equation}

\medskip 
{\bf Case 2:}   $\theta\in [b_n, 0]$. By
Lemma~\ref{operpoisson} with $v=r_n e^{-i\theta}, v'=-ie^{-i\theta}, C=r(b_n)e^{-i\theta}, R=r(b_n)-r_n$, 
\begin{align*}
&{\mu(\sigma_n(\theta), T)+
\frac{r_n}{2\pi}\,R_n(\tht)^{-1}
}  \\
=&\frac{r_n}{2\pi}\,(r(b_n)-r_n)A_2^{-1}B_2\bigg(\frac{1}{(r(b_n)-r_n)^2}-\big(T-r(b_n)e^{-i\theta}\big)^{-1}\big(T-r(b_n)e^{-i\theta}\big)^{-\ast}\bigg)B_2^*A_2^{-\ast}\\ 
=&P_n(\theta),
\end{align*}
where 
$R_n(\tht)=r(b_n)-r_n$, 
$A_2=r_ne^{-i\theta}-T, B_2=T-r(b_n)e^{-i\theta}$ and $P_n(\theta)\succeq 0$ because, by ~\eqref{anbn},
\[
\frac{1}{(r(b_n)-r_n)^2}=\frac{\big[1+\varepsilon(r(b_n)-1)\big]^2}{(r(b_n)-1)^2}
\]
and, for $\theta\in [b_n, 0]$, $r(b_n)e^{-i\tht}$ belongs to the domain defined by the v-like curve $\ga$ (due to the monotonicity of $r$). We put  $\fD_n(\tht)=0$ for these values of $\tht$. 

\medskip 
{\bf Case 3:}  
 $\theta\in [0, a_n]$. As in the previous case, it can be shown that 
\[
P_n(\tht):=\mu(\sigma_n(\theta), T)+ 
\frac {r_n}{2\pi}\, R_n(\tht)^{-1} \succeq 0, 
\]
{
where $R_n(\tht)=r(a_n)-r_n$ for these values of $\tht$. We again put $\fD_n(\tht)=0$.
} \par


%
%
%
{
Notice that for $n\ge 1,$  $R_n: [-\pi, \pi]\to (0,+\infty)$ has been defined by} 
\[
R_n(\theta)=
\begin{cases}
	r(\theta)-r_n,  
	    &\theta\in [-\pi, b_n]\cup[a_n, \pi] \\
	r(b_n)-r_n,  
	  &\theta\in [b_n, 0]\\
	r(a_n)-r_n,  
	    &\theta\in [0, a_n]\, .
\end{cases} 
\]
By our previous calculations,  both assertions (a) and (b) of Lemma~\ref{lem:Pn-Dn} hold. To verify assertion (c), we need some bounds for $\limsup_{n\to\infty}\int_{-\pi}^{\pi} R_n(\theta)^{-1}\ d\theta$. We may use \eqref{anbn} to write
%
%
\begin{align}
\int_{-\pi}^{b_n} R_n(\theta)^{-1} \, d\theta &=\int_{-\pi}^{b_n}\frac{1}{r(\theta)-r_n}\ d\theta	 \notag \\	
&=\int_{-\pi}^{b_n}\frac{1}{r(\theta)-1}\ d\theta+\int_{-\pi}^{b_n}\frac{r_n-1}{(r(\theta)-1)(r(\theta)-r_n)}\ d\theta  \notag \\  
&\le \int_{-\pi}^{b_n}\frac{1}{r(\theta)-1}\ d\theta    +\int_{-\pi}^{b_n}\frac{r_n-1}{(r(\theta)-1)(r(b_n)-r_n)}\ d\theta \notag \\ 
&= \int_{-\pi}^{b_n}\frac{1}{r(\theta)-1}\ d\theta    +\int_{-\pi}^{b_n}\frac{(r_n-1)[1+\varepsilon(r(b_n)-1)]}{(r(b_n)-1)(r(\theta)-1)}\ d\theta \notag \\ 
&\le \int_{-\pi}^{b_n}\frac{1}{r(\theta)-1}\ d\theta    +\int_{-\pi}^{b_n}\frac{1+\varepsilon(r(b_n)-1)}{r(\theta)-1}\ d\theta \notag,
\end{align}
with an identical estimate for $\int_{a_n}^{\pi}R_n^{-1}$. We also have
\[\int_{b_n}^{0} R_n(\theta)^{-1}\ d\theta=\int_{b_n}^{0} \frac{1+\varepsilon(r(b_n)-1)}{r(b_n)-1} d\theta, \]
and similarly for $\int_0^{a_n}R_n^{-1}.$ Thus,
\begin{multline} \label{ellnest}
\limsup_{n\to\infty}\int_{-\pi}^{\pi} R_n(\theta)^{-1}  d\theta=\limsup_{n\to\infty}
\bigg(\int_{I_{1n}}R_n(\theta)^{-1} d\theta+ 
  \int_{I_{2n}}R_n(\theta)^{-1} d\theta\bigg)  \\ 
\le  \int_{-\pi}^{\pi}\cfrac{2}{ r(\theta)-1}\; d\theta+2\limsup\limits_{\theta\to 0}\;\frac{|\theta|}{r(\theta)-1}\, , 
\end{multline}
which is finite because of our assumptions on $\gamma$ (using also ~\eqref{limsup=0}). 
\par 
%
It remains to verify assertion (d) of Lemma~\ref{lem:Pn-Dn}; for this, we have to estimate
 $\int_{I_{1n}}\|\mathfrak{D}_n(\theta)\| \ d\theta$. For $ n\ge 1,$ consider the function
 \[G_n(\theta):=r(\theta)-r_n-\cfrac{r(\theta)-1}{1+\epsilon(r(\theta)-1)},\]
 for $\theta\in(-\pi, \pi].$ From \eqref{r1}, we have $G_n(\pi)\ge G_1(\pi)> 0.$ Also, the definition of $a_n, b_n$ implies that they are, respectively, the largest and smallest solutions  of $G_n(\theta)=0$ in $[-\pi, \pi]$. Thus, we have $G_n(\theta)\ge 0$ for all $\theta \in I_{1n}$, which yields
%
%
%
%
\begin{align*}
0&\le 
\frac{\big[1+\varepsilon(r(\theta)-1)\big]^2}{(r(\theta)-1)^2} 
-
\frac{1}{(r(\theta)-r_n)^2}            \\
& \le  
\frac{\big[1+\varepsilon(r(\theta)-1)\big]^2}{(r(\theta)-1)^2} 
-
\frac{1}{(r(\theta)-1)^2}             \\
& \le 
\frac{\varepsilon(r(\theta)-1)^2+2\varepsilon(r(\theta)-1)}{(r(\theta)-1)^2}\; .
\end{align*}
By \eqref{dif2}, we may now write
\begin{align}
    &\limsup\limits_{n\to\infty}\int_{I_{1n}}\|\mathfrak{D}_n(\theta)\| \ d\theta \notag  \\
    \le & \limsup\limits_{n\to\infty}\int_{I_{1n}}\frac{r_n(r(\theta)-r_n)\big[\varepsilon(r(\theta)-1)^2+2\varepsilon(r(\theta)-1)\big]}{(r(\theta)-1)^2} \|(r_ne^{-i\theta}-T)^{-1}\|^2\|T-r(\theta)e^{-i\theta}\|^2 \ d\theta \notag \\
    \le &\ C\limsup\limits_{n\to\infty}  \int_{I_{1n}}r_n\cfrac{\varepsilon(r(\theta)-1)}{r(\theta)-1}\,\|(r_ne^{-i\theta}-T)^{-1}\|^2\ d\theta, \label{wes}   
\end{align}
where $C$ is a constant that only depends on $\|T\|$ and on $\max_x |\varepsilon(x)|$ and $\max_{\theta}r(\theta)$. Now, observe that, by the resolvent identity and the fact that $T$ satisfies the Kreiss condition,
\begin{align*}
 \|(r_ne^{-i\theta}-T)^{-1}-(e^{-i\theta}-T)^{-1}\|&=\|e^{-i\theta}(r_n-1)(r_ne^{-i\theta}-T)^{-1}(e^{-i\theta}-T)^{-1}\|    \\
\ & \le C_1\|(e^{-i\theta}-T)^{-1}\|,
\end{align*}
where $C_1$ is the Kreiss constant of $T.$ 
Thus, we may write 
\begin{align*}
&\|(r_ne^{-i\theta}-T)^{-1}\|^2\le \Big[
\big\|(r_ne^{-i\theta}-T)^{-1}-(e^{-i\theta}-T)^{-1}\big\|+
\big\|(e^{-i\theta}-T)^{-1}\big\|\Big]^2 \\ 
&\le \  C_1^2\|(e^{-i\theta}-T)^{-1}\|^2+2C_1\|(e^{-i\theta}-T)^{-1}\|^2+\|(e^{-i\theta}-T)^{-1}\|^2,
\end{align*}
so \eqref{wes} yields 
\begin{equation}\label{D2est}
\limsup\limits_{n\to\infty} \int_{I_{1n}}\|\fD_n(\theta)\| \ d\theta 
	\le\  C_2 \int_{-\pi}^{\pi}  \cfrac{\varepsilon(r(\theta)-1)}{r(\theta)-1}\,\|(e^{-i\theta}-T)^{-1}\|^2\ d\theta<\infty, 
\end{equation}
where $C_2$ is   a constant that depends on $C_1$, $\|T\|$, $\max_x |\varepsilon(x)|$ and $\max_{\theta}r(\theta)$. 
\end{proof}

\section{Counterexamples} \label{counterexsection}
\subsection{\eqref{eps-resolv-estim} does not imply power-boundedness}

\begin{theorem}\label{resolvcounter}
    There exists $T\in\mathcal{B}(H)$ and a continuous $\eps:(0, \nu)\to[0,\infty)$ with $\eps(0^+)=0$ such that \eqref{eps-resolv-estim} holds and $T$ is not power bounded.
\end{theorem}

\begin{proof}
Given $c>0$, \eqref{decompult} and Lemma \ref{anotherprelimyes} tell us that there exists a constant $L>0$ such that, for any $|z|>1,$
\[\big\| (zI-T_{1, c})^{-1}- (zI-S)^{-1}  \big\|\le \cfrac{cL}{|z|-1}.\]
Let $N\ge 2$ and recall 
$X_{m, N}(c)=P_NT_{m, c}\big|_{V_n}=T_{m, c}\big|_{V_n}$. Conjugating $(zI-T_{1, c})^{-1}- (zI-S)^{-1}$ by the projection $P_N$ yields
\begin{equation} \label{truncX}
    \Big\| (zI_{V_N}-X_{1, N}(c))^{-1}- (zI_{V_N}-P_NS\big|_{V_N})^{-1}  \Big\|\le \cfrac{cL}{|z|-1}. 
\end{equation}
Since the operator $S_N:=P_NS\big|_{V_N}$ satisfies $S_N^N=0$, we may use \eqref{truncX} to obtain
\begin{align}
 \big\| (zP_N-X_{1, N}(c))^{-1}\big\|&\le |z|^{-1}\bigg\|\sum_{j=0}^{N-1}\cfrac{S^j_N}{z^j}\bigg\| + \cfrac{cL}{|z|-1} \notag \\
&\le  \cfrac{1+cL-|z|^{-N}}{|z|-1},         \label{truncX2}
\end{align}
 for all $|z|>1.$
Now, let $c_N>0$ be any decreasing null sequence such that $c_N\log N\to \infty. $ Set $Y_{N}:=X_{1, N}(c_N)$, choose $N_0\ge 2$ such that $c_N L<1$ for $N\ge N_0$ and define 
\[Y:=\text{diag}\{Y_{ N}\}_{N\ge N_0}. \]
Observe that \eqref{penult} gives us 
\[\|Y^{N-1}\|\ge \|Y^{N-1}_{N}\|\ge c_N\log N,\]
and thus $Y$ is not power-bounded. 
Further, define $\eps:(0, +\infty)\to [-1,1]$ by 
\[
	\eps(x):=\sup_{N\ge N_0}\{c_NL-(1+x)^{-N}\}. 
\]
%
%
Clearly, $\eps$ is an increasing function. Let us show that $\eps(0^+)\le 0$. Let $\{x_k\}$ be an arbitrary strictly decreasing sequence tending to $0.$ Choose a sequence $\{N_k\}$ such that $|c_{N_k}L-(1+x_k)^{-{N_k}}-\eps(x_k)|\to 0$ 
as $k\to\infty.$ If $\{N_k\}$ is bounded, then
\[
\limsup_k \big\{c_{N_k}L-(1+x_k)^{-{N_k}}\big\}\le 1-1=0. \]
On the other hand, if $\{N_k\}$ goes to infinity (after passing to a subsequence if necessary), then we obtain 
\begin{align*}
    \limsup_k \big\{c_{N_k}L-(1+x_k)^{-{N_k}}\big\}&=\lim_k(c_{N_k}L)-\liminf_k\{(1+x_k)^{-N_k}\} \\ &=-\liminf_k\{(1+x_k)^{-N_k}\} \\ 
    &\le 0.
\end{align*}
Thus, $\eps(0^+)\le0$. Now, \eqref{truncX2} yields, for $|z|>1$, that 
\begin{align*}
 \|(zI-Y)^{-1}\|&=\sup_{N\ge N_0}\|(zI-Y_N)^{-1}\|   \\ 
   &\le\sup_{N\ge N_0}\bigg\{ \cfrac{1+c_NL-|z|^{-N}}{|z|-1}\bigg\}                        \\ 
   &=\cfrac{1+\eps(|z|-1)}{|z|-1}. 
\end{align*}
In particular, the spectrum of $Y$ is contained in 
$\{|z|\le 1\}$. 
Since $Y$ is not power bounded, its spectral radius 
is equal to $1$. 
Hence, $\eps(x)\ge 0$ for any $x>0$. It follows that 
$\eps(0^+)=0$. It is now easy to find a continuous increasing piecewise linear function $\eps_1\ge \eps$, defined on an interval $(0,\delta]$, such that 
one also has $\eps_1(0^+)=0$ (put $\eps_1(x_{k+1})=\eps(x_k)$ for $k\ge 1$
and extend $\eps_1$ linearly to each 
interval $[x_{k+1}, x_k]$, $k\ge2$). Replacing $\eps$ by $\eps_1$ completes the proof.
\end{proof}

\subsection{A Pisier-type counterexample}

To prove Theorem~\ref{thm-Pisier-counterex},
we will adapt a variant of Pisier's counterexample, given in \cite[Chapter 10]{Paulsenbook}. Namely, denote by $S$ the forward shift on $\ell^2$: 
$S(x_0,x_1,\dots)=(0,x_0,x_1,\dots)$. 
Next, let $K$ be a Hilbert space 
and denote by $S_K$ the forward shift on $\ell^2(K)$: 
$S_K(k_0,k_1,\dots)=(0,k_0,k_1,\dots)$, whose adjoint 
acts as $S^*_K(k_0,k_1,\dots)=(k_1,k_2,\dots)$. 
We will take as $F$ the following Foguel-Hankel 
operator
\begin{equation} \label{eq:Foguel-Hankel}
F=\begin{bmatrix}
	S^*_K & X \\ 
	0 & S_K
\end{bmatrix}. 
\end{equation}
Here $X=(a_{i+j}W_{i+j})$, where $W_n\in \mathcal{B}(K)$ are chosen as in \cite[page 142]{Paulsenbook} and $a_n$ are complex numbers, whose choice will be specified later. In particular, $\{W_n\}$ satisfy the CAR relations modulo the ideal of compact operators. 
We set $X_n=nXS_K^{n-1}$ to be the $(1,2)$-entry of $F^n$. We summarize the facts which we will use as follows. 

\begin{lemma}[\cite{Paulsenbook}]
	\label{lem:Paulsenbook}
	\begin{itemize}
	 \item[(i)] 
	 If $\sum_{k=0}^\infty (k+1)^2 |a_k|^2$ is infinite, then $F$ is not similar to a contraction; 
	 
	 \item[(ii)] One has 
	 $
	 \|X_n\|^2\le n^2\sum_{k=0}^\infty |a_{k+n-1}|^2 \le  2\|X_n\|^2. 
	 $
	\end{itemize}	
\end{lemma}

\begin{proof} Assertion (i) is contained in 
	\cite[Theorem 10.5]{Paulsenbook}. 
	
	The right-hand inequality in (ii) follows from a calculation in the proof of this theorem (\cite{Paulsenbook}, page 143). Notice the typo in item (iii) in \cite[Theorem 10.5]{Paulsenbook}; it should read as $\sup_n n^2\sum_{k=n-1}^\infty |a_k|^2<\infty$. 
	
	To get the left-hand inequality in (ii), let $\{E_{i,j}\}_{i,j=0}^\infty$ be the usual matrix units, considered as linear operators on (scalar) $\ell^2$. 
	Set $\wt X=(a_{i+j} E_{i+j,0})$. Following 
	the proof of Theorem 10.5 in \cite{Paulsenbook}, define 
	$\Phi:B(\ell^2)\to B(K)$ by 
	$\Phi((a_{i,j}))=\sum_{i=0}^\infty a_{i,0} W_i$. 
	Then $\|\Phi\|\le 1$. 
	Put $\wt X_n=n\wt XS^{n-1}$. Then 
	$\Phi(\wt X)=X$ and $\Phi(\wt X_n)=X_n$, so that we get 
	\[
	\|X_n\|^2=
	\|\Phi(\wt X_n)\|^2\le 
	\|\wt X_n\|^2 = n^2\sum_{k=n-1}^\infty |a_k|^2. \qedhere 
	\]
\end{proof}

\begin{proof}[Proof of Theorem~\ref{thm-Pisier-counterex}]
Since $\{n^{-2}\be_n\}$ decreases, we can take a sequence $\{a_n\}$ such that 
\[
|a_n|^2=\frac{\be_{n+1}}{(n+1)^2}-\frac{\be_{n+2}}{(n+2)^2}
\]
for all $n\ge 0$. 
Then 
$\sum_{n=0}^\infty (n+1)^2|a_n|^2=
\be_1+\sum_{n=0}^\infty\big(1-(\frac{n+1}{n+2})^2\big)\be_{n+2}
=\infty$. 
Hence, by Lemma~\ref{lem:Paulsenbook}, 
the corresponding Foguel-Hankel operator $F$ is not similar to a contraction. 
On the other hand, by the same Lemma,  
\[
\|X_n\|^2\le n^2\sum_{k=n-1}^\infty |a_k|^2 
=n^2\sum_{k=n-1}^\infty \Big(\frac{\be_{k+1}}{(k+1)^2}-\frac{\be_{k+2}}{(k+2)^2}\Big)
= \be_{n}.
\]
Since $\|F^n\|\le 1+\|X_n\|$, the proof is complete. 
\end{proof}

\begin{remark} \label{epsconversion}
Now, assume $F$ satisfies the hypotheses of Theorem \ref{thm-Pisier-counterex} with $\beta_n=1/\log(n+1)$. Given $0<a<1,$ set $k=\lfloor(1-a)^{-1}\rfloor$. Observe that
\begin{align}
\sum^{\infty}_{n=2}\cfrac{a^n}{\sqrt{\log n}}&=\sum_{n=2}^k\cfrac{a^n}{\sqrt{\log n}}+\sum_{n=k+1}^{\infty}\cfrac{a^n}{\sqrt{\log n}} \notag \\
&\le \sum_{n=2}^k\cfrac{1}{\sqrt{\log n}}+\cfrac{1}{\sqrt{\log k}}\cdot \cfrac{a^{k+1}}{1-a}  \notag \\
&\lesssim \int^k_2\cfrac{1}{\sqrt{\log a}}\ da +
\cfrac{|\log(1-a)|^{-1/2}}{1-a}.
\notag
\end{align}
Since $\int^k_2\frac{1}{\sqrt{\log a}}\ da \approx \frac{k}{\sqrt{\log k}}$, the previous inequality yields
\begin{align}
\sum^{\infty}_{n=2}\cfrac{a^n}{\sqrt{\log n}}& \lesssim  
\frac{k}{\sqrt{\log k}}+ \cfrac{|\log(1-a)|^{-1/2}}{1-a} \notag \\
& \lesssim \cfrac{|\log(1-a)|^{-1/2}}{1-a}.\label{aux}
\end{align}
Thus, there exists a numerical constant $C$ such that, for $1<|\lambda|<1+\nu$, 
\begin{align*}
\|(\lambda-F)^{-1}\|&\le \sum_{n=0}^{\infty}\cfrac{\|F\|^{n}}{|\lambda|^{n+1}} \\ 
&\le \cfrac{1}{|\lambda|-1}+\sum_{n=1}\cfrac{|\lambda|^{-n-1}}{\sqrt{\log(n+1)}} \\ 
&\le \cfrac{1}{|\lambda|-1}+C\,\cfrac{\big|\log (|\lambda|-1)\big|^{-1/2}}{|\lambda|-1},
\end{align*}
where the last inequality is due to \eqref{aux}, with $a$ replaced by $1/|\lambda|$. In other words, $F$ satisfies \eqref{eps-resolv-estim} with 
$\eps(x)=C\big(\log (1/x)\big)^{-1/2}$.
\end{remark}

\section{Open Questions} \label{questions}
Nikolski conjectured \cite{Nikolski} that one should always be able to obtain sublinear growth in $N$ in \eqref{spijk} as long as the Kreiss constant 
remains bounded. In fact, he proved that 
if $A$ is a diagonalizable $N\times N$ matrix 
with unimodular eigenvalues, then there exists $\epsilon>0$ (depending on the basis constant of the eigenvectors of $A$) such that 
\[\sup_{n\ge 1}\|A^n\|\le 2\pi K N^{1-\epsilon}.\]
Unfortunately, it is not clear how exactly $\epsilon$ depends on $K$ here (see also \cite[Theorem 2.9.9]{Nikolski}), while the case of general $N\times N$ matrices currently lies out of reach of Nikolski's methods. As such,  there remains a sizeable gap between the lower bounds \eqref{nik}-\eqref{main} and Spijker's estimate \eqref{spijk}, which naturally motivates the following questions:

\begin{question}\label{q1}
Does there exist a strictly increasing function $\alpha(\cdot): (1, \infty)\to (0, 1)$ with $\lim_{K\to 1+}\alpha(K)=0$ and $\lim_{K\to \infty }\alpha(K)=1$ such that 
\[P(N, K)\gtrsim_K N^{\alpha(K)},\]
for all $N\ge 1$ and $K>1$?
\end{question}
\begin{question}\label{q2}
  Does there exist a constant $\delta>0$ such that \[P(N, K)\gtrsim_K N^{\delta}, \]
  for all $N\ge 1$ and $K>1$?
\end{question}
\noindent  Inequality \eqref{main} asserts that the answer to Question \ref{q2} is yes if the growth function $N^{\delta}$ is replaced by $\log^m (N),$  independently of the value of $m\ge 1.$ \par
Regarding condition \eqref{eps-resolv-estim} on the other hand, the following questions remain unanswered.
\begin{question}\label{q3}
  Does there exist a continuous function $\eps: (0, \nu]\to (0, \infty)$ such that \eqref{eps-resolv-estim} guarantees power-boundedness? What if $\sigma(T)\cap\mathbb{T}=\{1\}?$
\end{question}

\begin{question}\label{q4}
Does there exist a decreasing positive sequence $\{b_k\}$, tending to zero, such that $\|T^k\|\le 1+b_k$ for all $k$ implies the similarity of $T$ to a contraction?
\end{question}

\printbibliography

\end{document}